\documentclass[12pt, twoside]{amsart}
\usepackage{a4, amssymb, amsmath, amsthm, latexsym, cite, hyperref}
\usepackage{graphicx, amsfonts, ulem, float}
\usepackage{tikz,tikz-network,pgfplots} 
\usetikzlibrary{shapes, arrows, positioning}
\usepackage[font=small, labelsep=none, figurename=Fig.]{caption}
\newtheorem{theorem}{Theorem}[section]

\newtheorem{definition}[theorem]{Definition}
\newtheorem{example}[theorem]{Example}
\newtheorem{lemma} [theorem]{Lemma}
\newtheorem{notation}[theorem]{Notation}

\newtheorem{remark}[theorem]{Remark}

\vspace{5cm}

\begin{document}
\title{\bf Energy and Adjacency Spectra of Semigraphs}
\author{Pralhad M. Shinde$^{1}$, Charusheela Deshpande$^{2}$ }
\address{$^{1,2}$Dept. of Maths, College of Engineering Pune, Maharashtra-411005, India.}
\email{$^1$pralhadmohanshinde@gmail.com, $^2$dcm.maths@coep.ac.in}
\maketitle

\thispagestyle{empty}

\begin{abstract}	
{\footnotesize  In this paper, we study the energy of semigraphs and obtain some bounds, and show that one of the bounds is tight. We also study the spectra of the adjacency matrix of a special type of rooted 3-uniform semigraph and enumerate those explicitly. 
}
 \end{abstract}



\vskip 1cm
\section{Introduction}
If $G$ is a graph of order $n$ and $A$ is its adjacency matrix then energy of graph $G$, denoted by $\mathcal{E}(G)$, is the sum of absolute values of its eigenvalues. Energy of graph is extensively studied in literature \cite{gutman}. In \cite{cmd}, authors studied the energy of semigraphs for their adjacency matrix which is not a symmetric matrix. In \cite{pms}, the first author defined the adjacency matrix which enjoys the symmetric property and showed that various spectral graph theory results can be extended to semigraphs. And it draws the inspiration to study the energy of semigraph with the new definition of adjacency matrix. 

This paper is organized as follows. In Section 2, we  study eigenvalues of special type of rooted 3-uniform semigraph.  In Section 3, we define the energy of a semigraph and give lower and upper bounds for the same. Further, we enumerate energies of some special type of semigraphs and compare them with the bounds obtained.

\section*{Preliminaries}
For all basic definitions and standard notations please refer to \cite{pms}, \cite{smt}.\\
We recall some definitions here;
\begin{definition} \label{def:1}
Let $V$ be a non-empty set having $n$ elements. A $semigraph$ is a pair G=$(V, E),$ where the elements of $V$are called vertices and $E$ is a set of ordered $k$-tuples of distinct vertices, whose elements are called edges of $G$; for $n\geq 2,$ satisfying the following conditions:

\begin{enumerate}
	\item  Any two edges have at most one vertex in common
	\item Two edges $(u_1,u_2,\cdots,u_k)$ and $(v_1,v_2,\cdots,v_r)$ are considered to be equal if
	\begin{enumerate}
	\item r = k and
	\item either $u_i = v_i$ for $1\leq i \leq k,$  or $u_i=v_{k-i+1}$, for $1\leq i \leq k.$ 
	\end{enumerate}
\end{enumerate}
\end{definition}

Thus the edge $(u_1,u_2,\cdots,u_r)$ is same as
$(u_r,u_{r-1},\cdots,u_1)$.\\\\
Two vertices $v_i,\; v_j$ in a semigraph are said to be {\bf\itshape adjacent} if they belong to the same edge  and are said to be {\bf\itshape consecutively adjacent} if in addition they are consecutive in order as well. 
%We represent the adjacency of $v_i, v_j$ by $v_i \sim_{e} v_j$ and consecutive adjacency by $v_i \sim_{s} v_j$.
\vskip2mm
\noindent For the edge $e=(u_1,u_2,\cdots,u_n)$, $u_1$ and $u_n$ are called the {\bf \it end} vertices of $e$ and $u_2,u_3,\cdots,u_{n-1}$ are called the {\bf \it middle} vertices of $e$. Note that $u_i, \; u_j$ are adjacent for all $1\leq i, j\leq n$ while $u_i, u_{i+1}$ are consecutively adjacent for all $1\leq i\leq n-1$. \\

For a semigraph, we define following types of vertices and edges:
\begin{enumerate}
\item $u_i$ is said to be a pure end vertex if it is an end vertex of every edge to which it belongs. 
\item $u_i$ is said to be a pure middle vertex if it is a middle vertex of every edge to which it belongs. 
\item $u_i$ is said to be a middle end vertex if it is middle vertex of at least one edge and end vertex of at least one other edge.
\item An edge $e=(u_1, u_2, \cdots, u_k),\; k\geq2$ is said to be full edge if $u_1$ and $u_k$ are pure end vertices.
\item An edge $e=(u_1, u_2, \cdots, u_k),\; k>2$ is said to be an half edge if either $u_1$ or $u_k$ (or both) are middle end vertices.
\item An edge $e=(u_1, u_2)$ is said to be a quarter edge if both  $u_1$ and $u_2$ are middle end vertices while $e=(u_1, u_2)$ will be  half edge if exactly one of $u_1$ and $u_2$ is a middle end vertex and other is a pure end vertex. 
\end{enumerate}
 \par For a full edge $e=(u_1, u_2, \cdots, u_k)$, $(u_i, u_{i+1})\; \forall\; 1\leq i\leq k-1$ is called a partial edge of $e$ while for a half edge $e=(u_1, u_2, \cdots, u_{k-1}, u_k)$, $(u_1, u_2)$ is called partial half edge if $u_1$ is middle end vertex,  $(u_{k-1}, u_k)$ is a partial half edge if $u_k$ is middle end vertex and $(u_i, u_{i+1})\; \forall\; 2\leq i\leq k-2$ are partial edges. Thus, any half edge can have at most two partial edges. 
\vskip2mm
\begin{example}
Let $G = (V, E)$ be a semigraph, with $V = \{v_1,v_2,\cdots,v_{10}\}$ as a vertex set and 
$E = \{(v_1,v_2,v_3,v_4,v_5),$ $(v_1,v_7,v_8),$ $(v_2,v_6,v_8),$ $(v_1,v_9),$ $(v_6,v_7)\}.$ as an edge set.
\end{example}

\begin{figure}[h]
\centering
 \begin{tikzpicture}[yscale=0.5]
  \Vertex[x=-1, y=2, size=0.2,  label=$v_9$, position=left, color=black]{I} 
 \Vertex[size=0.2,  label=$v_1$, position=below, color=black]{A} 
  \Vertex[x=3, size=0.2,label=$v_2$, color=none, position=below]{B} 
  \Vertex[x=5,size=0.2,label=$v_3$,position=below, color=none]{C} 
   \Vertex[x=7, size=0.2, label=$v_4$,position=below, color=none]{D}
   \Vertex[x=9, size=0.2, label=$v_5$,position=below, color=black]{J}
  \Vertex[x=3,y=2,size=0.2,label=$v_6$,position=right, color=none]{E}  
  \Vertex[x=3, y=4, size=0.2, label=$v_8$,position=above, color=black]{F}
  \Vertex[x=1.5,y=2,size=0.2,label=$v_7$,position=left, color=none]{G}  
    \Vertex[x=5,y=3,size=0.2,label=$v_{10}$,position=right, color=black]{H}  
  \Edge(A)(B) \Edge(B)(C) \Edge(C)(D) \Edge(A)(G) \Edge(G)(F) \Edge(F)(E)  \Edge(I)(A) \Edge(J)(D)
  \draw[thick](2.9,0.3)--(3.1, 0.3);
  \draw[thick](1.65,2)--(2.85, 2);
  \draw[thick](3, 1.85)--(3, 0.3);
  \draw[thick](2.85,1.75)--(2.85, 2.2);
 \draw[thick](1.65,1.75)--(1.65, 2.2);
\end{tikzpicture}
\caption{}
\label{fig:1}
 \end{figure}

In Fig.~\ref{fig:1}, vertices $v_1,v_5,v_8,$and$v_9$ are the pure end vertices; $v_3, v_4$ are pure middle vertices; $v_2$,$v_6,$and$v_7$ are the middle end vertices and $v_{10}$ is an isolated vertex. Further, $(v_1, v_9)$, $(v_1, v_7, v_8)$, $(v_1, v_2, v_3, v_4, v_5)$ are full edges whereas $(v_2, v_6, v_8)$ is an half edge with only $(v_2, v_6)$ as a partial half edge. Note that $(v_6, v_7)$ is the a quarter edge. 
 \begin{definition}  \label{def:2}
 A semigraph $G=(V, E)$ is said to be connected if for any two vertices $u,\; v\;\in E$, there exist a sequence of edges $e_{i_1},\cdots, e_{i_p}$ for some $p$ such that $u\in e_{i_1},\; v\in e_{i_p}$ and $|e_{i_{j}}\cap e_{i_{j+1}}|=1,\; \forall\; 1\leq j\leq p-1$.
 \end{definition}

\begin{notation}
Throughout this paper, we assume that semigraph is connected and $G=(V, E)$ denotes the semigraph with $n$ vertices and $m$ edges such that
\begin{itemize}
\item $m_1$ is the number of full edges 
\item $m_2$ is the number quarter edges
\item $m_3$ is the number of half edges with one partial half edge
\item $m_4$ is the number of half edges with two partial half edges  
\end{itemize}  
Note that $m=m_1+m_2+m_3+m_4$ and if $G$ is a graph then $m_2=m_3=m_4=0$ and $m=m_1$.
\end{notation}
\vskip2mm

 \section{Adjacency matrix}
  Let $G$=$(V, E)$ be a semigraph, with $V=\{v_1, v_2,\cdots, v_n\}$ as a vertex set and $E=\{e_1, e_2,\cdots, e_m\}$ as an edge set.  
Recall that the graph skeleton~\cite[definition 1.5]{pms} of $G$ is an underlined graph structure $G^{S}$ of the semigraph on $V$, where two vertices $v_{i}$, $v_{j}$ are adjacent in $G^{S}$ iff $v_{i}$ and $v_{j}$ are consecutively adjacent in $G.$   
 Let $u_i, u_j \in e=(u_1, u_2,\cdots, u_k)$ for some $e \in E$. Let $d_{e}(u_i,u_j)$ denote the distance between $u_i$ and $u_j$ in the graph skeleton of $e$. The distance $d_{e}(u_i,u_j)$ is well-defined as each pair of vertices in semigraph belongs to at most one edge. 
  
  \begin{definition}~\cite{pms} \label{def:4}
  We index the rows and columns of a matrix $A=(a_{ij})_{n\times n}$ by vertices $ v_1, v_2,\cdots, v_n,$ where $a_{ij}$ is given as follows:
$$a_{ij}=\begin{cases}
 d_{e}(v_i,v_j),& \text{if $v_i,\; v_{j}$ belong to a full edge or a half edge such that } \\
 &\text{$(v_i, v_j)$ is neither a partial half edge nor a quarter edge}\\
\;\;\; \frac{1}{2}, & \text{if $(v_i,\;  v_{j})$ is a partial half edge}\\
\;\;\;\frac{1}{4}, &\text{if $(v_i,  v_{j})$ is a quarter edge}\\
\;\;\;0,&\text{otherwise}
\end{cases}$$ 
\end{definition}
The matrix $A=(a_{ij})_{n\times n}$ is called the adjacency matrix of semigraph $G.$ 

Let $A_{i}$ be the $i^{th}$ row of the adjacency matrix $A$, we define the degree of vertices in semigraph as $d_i=A_{i}\mathbf{1}$, $\mathbf{1}$ being a column matrix with all entries 1. 
 
 \section{Spectra of rooted 3-uniform semigraph tree}
 In this section, we compute the eigenvalues of the rooted 3-uniform semigraph tree. Let $T^3_n$ denote the semigraph on $2n+1$ vertices with $n$ edges. The edge set $E$ is given by $\{(v_1, v_{2i},v_{2i+1})\; \big{|}\; 1\leq i\leq n \}$
\begin{figure}[h]
\centering
  \begin{tikzpicture}[yscale=0.5]
 \Vertex[size=0.2, label=$v_1$, position=above, color=black]{A} 
  \Vertex[size=0.2, x=1, y=-1,label=$v_{2n}$, position=right, color=none]{B} 
   \Vertex[size=0.2, x=2, y=-2,label=$v_{2n+1}$, position=below, color=black]{C}
   \Vertex[size=0.2, x=-1, y=-1,label=$v_2$, position=left, color=none]{D} 
   \Vertex[size=0.2, x=-2, y=-2,label=$v_3$, position=below, color=black]{E}
  \Vertex[size=0.2, x=0, y=-2,label=$v_{2r}$, position=left, color=none]{F} 
   \Vertex[size=0.2, x=0, y=-4,label=$v_{2r+1}$, position=below, color=black]{G} 
   \Edge(A)(D) \Edge(D)(E) \Edge(A)(F) \Edge(F)(G) \Edge(A)(B) \Edge(B)(C)
 
 \Edge[bend =-20, style={dashed}](E)(G)
  \Edge[bend =-20, style={dashed}](G)(C)
\end{tikzpicture}\\
$T^3_{n} $ 
\caption{}
\label{fig:2}
\end{figure}

The adjacency matrix $A$ of rooted 3-uniform semigraph tree: $T^{3}_{n}$ is
\[
     \bordermatrix{ & {v_1} & {v_2} & {v_3}& {v_4} & {v_5}  & \cdots  & {v_{2n-2}}& {v_{2n-1}} &{v_{2n}}& {v_{2n+1}} \cr
       v_1 & 0&1&2&1&2&\cdots&1&2&1&2 \cr
       v_2 & 1&0&1&0&0&\cdots&0&0&0&0\cr
       v_3 & 2&1&0&0&0&\cdots&0&0&0&0 \cr
        v_4 & 1&0&0&0&1&\cdots&0&0&0&0 \cr
       v_5 & 2&0&0&1&0&\cdots&0&0&0&0 \cr
   \vdots & \vdots&\vdots&\vdots&\vdots&\vdots&\ddots&\vdots&\vdots&\vdots&\vdots\cr
       v_{2n-2} &1&0&0&0&0&\cdots&0&1&0&0 \cr
        v_{2n-1} &2&0&0&0&0&\cdots&1&0&0&0 \cr
       v_{2n} &1&0&0&0&0&\cdots&0&0&0&1 \cr
       v_{2n+1} &2&0&0&0&0&\cdots&0&0&1&0} \qquad
 \]

 \begin{lemma}\label{lemma:1}
 The spectrum of $T^3_{n}$ is: 
 %\[
 %\bordermatrix{&&&&&\cr
 %eigenvalues&-1&1&\lambda_1&\lambda_2&\lambda_3\cr
% multiplicity &n-1&n-1&1&1&1}\qquad
 %\]
 $$\begin{pmatrix}-1&1&\lambda_1& \lambda_2& \lambda_3 \\ n-1&n-1&1&1&1\end{pmatrix}$$ 
 where
   $\lambda_1, \lambda_2, \lambda_3$ are roots of $\lambda^3-(5n+1)\lambda-4n$.
\end{lemma}
\begin{proof}
Consider the characteristic polynomial $T_{n}(\lambda) =det(\lambda I -A)$

$$T_{n}(\lambda)=\begin{vmatrix}
\lambda&-1&-2&-1&-2&\cdots& -1&-2&-1&-2 \\
-1&\lambda&-1&0&0&\cdots & 0&0&0&0 \\
-2&-1&\lambda&0&0 &\cdots & 0&0&0&0 \\
-1&0&0&\lambda&-1 &\cdots &0&0&0&0 \\
-2&0&0&-1&\lambda & \cdots &0&0&0&0 \\
&&&\vdots&&\vdots&&&& \\
-1&0&0&0&0 & \cdots & \lambda&-1&0&0 \\
-2&0&0&0&0 & \cdots & -1&\lambda&0&0 \\
-1&0&0&0&0&\cdots& 0&0&\lambda&-1 \\
-2&0&0&0&0&\cdots&0&0&-1&\lambda
\end{vmatrix} $$
Using co-factor expansion along the last column, we get
\begin{align*}
T_{n}(\lambda)=&\lambda A_{(2n+1)(2n+1)} +1A_{(2n)(2n+1)} -2A_{1(2n+1)} 
\end{align*}
Further using the cofactor expansions of the above terms along the last columns and later taking the cofactor expansion along the last row if the only non-zero entry is in the first column-last row, we get
\small{
\begin{align*}
T_{n}=&\lambda^2 T_{n-1}-
\lambda (\lambda^2 -1)^{n-1}-T_{n-1}-2(\lambda^2-1)^{n-1}-2(\lambda^2-1)^{n-1}-4\lambda(\lambda^2-1)^{n-1}\\
T_n=&(\lambda^2-1)T_{n-1}-5\lambda(\lambda^2-1)^{n-1}-4(\lambda^2-1)^{n-1}
\end{align*}
 }
Applying the same formula for $T_{n-1}$ we get
$$T_{n-1}=(\lambda^2-1)T_{n-2}-5\lambda(\lambda^2-1)^{n-2}-4(\lambda^2-1)^{n-2}$$
Plugging this to the formula of $T_n$, we get
 $$T_n=(\lambda^2-1)\left[(\lambda^2-1)T_{n-1}-5\lambda(\lambda^2-1)^{n-1}-
4(\lambda^2-1)^{n-1}\right]-5\lambda(\lambda^2-1)^{n-1}-4(\lambda^2-1)^{n-1}$$
Simplifying it gives us
$$T_n=(\lambda^2-1)^2T_{n-2}-5\times 2\;\lambda(\lambda^2-1)^{n-1}-4\times 2\;(\lambda^2-1)^{n-1}$$
Continuing recursive substitution we get 
$$T_n=(\lambda^2-1)^{n-1}T_{1}-5(n-1)\;\lambda(\lambda^2-1)^{n-1}-4(n-1)\;(\lambda^2-1)^{n-1}$$
Here, $T_1$ is a single edge of length 3 and it's characteristics equation is $\lambda^3-6\lambda-4$. 
Therefore, 
$$T_n=(\lambda^2-1)^{n-1}(\lambda^3-6\lambda-4)-5(n-1)\;\lambda(\lambda^2-1)^{n-1}-4(n-1)\;(\lambda^2-1)^{n-1}$$
Further simplifying and re-arranging the terms we get
$$T_n(\lambda)=(\lambda^2-1)^{n-1}\;(\lambda^3-(5n+1)\lambda-4n)$$
Hence, the spectrum of $T^3_{n}$ is: -1 and 1 repeated $n-1$ times and roots $\lambda_1, \lambda_2, \lambda_3$ of $\lambda^3-(5n+1)\lambda-4n$.
\end{proof}
\vskip 1.5cm

\section{Bounds on energy}
Energy of semigraph is defined as the sum of absolute values of its eigenvalues. In \cite{cmd}, authors have calculated energy of semigraphs using their adjacency matrix. We use our definition of adjacency matrix to find the energy and also find some bounds. The bounds obtained here are similar to the energy bounds of graphs and it turns out as generalization of graph bounds. \\
To prove the theorem that follow we need the following definition of cartesian product of semigarphs and result on eigenvalues of the product.
\begin{definition}\cite{smt}
Let $G_1=(V_1, E_1)$, $G_2=(V_2, E_2)$ be two semigraphs with $|V_1|=m,\; |V_2|=n$. The cartesian product $G=G_1\times G_2$ of $G_1$ and $G_2$ is a semigraph on the vertex set $V=\{(u_i, v_j)\;|\;u_i\in V_1,\; v_j\in V_2\}$. The edges are of the form $\left((u_i, v_{i_1}),(u_i, v_{i_r}),\cdots, (u_i, v_{i_r})\right)$ for some edge $(v_{i_1}, v_{i_2},\cdots, v_{i_r})$ of $G_2$ or  $\left((u_{i_1}, v_j),(u_{i_2}, v_j),\cdots, (u_{i_p}, v_{j})\right)$  for some edge $(u_{i_1}, u_{i_2},\cdots, u_{i_p})$ of $G_1$. 
\end{definition}
Note that adjacency matrix of $G_1\times G_2$ is of form $I_m\otimes B+A\otimes I_n$, where $A_{m \times m}$ and $B_{n\times n}$ are adjacency matrices of $G_1$ and $G_2$ respectively. Thus, eigenvalues of $G=G_1\times G_2$ are of the form $\lambda_i+\mu_j$ where $\lambda_i,\; \mu_j$ are eigenvalues of $G_1$ and $G_2$ respectively.

\begin{theorem}
If $\mathcal{E}(G)$ is a rational number then it must be either an even integer or a rational number of the form $\frac{N}{2^r}$, for some $N\in \mathbb{Z}, \; r\in \mathbb{N}$ depending on whether $G$ has no middle end vertices or has middle end vertices.
\end{theorem}   
\begin{proof}
Let $\lambda_1,\lambda_2,\cdots,\lambda_n$ be eigenvalues of adjacency matrix $A$ of $G$. Trace of $A$ being $0$ implies $\displaystyle \sum_{i=1}^{n}\lambda_i=0$. Hence,$ \displaystyle \sum_{i=1}^{k}\lambda_i=\displaystyle \sum_{i=k+1}^{n}\lambda_i$, where $\lambda_1,\lambda_2,\cdots,\lambda_k$ are non-negative eigenvalues and $\lambda_{k+1},\lambda_{k+2},\cdots,\lambda_n$ are negative eigenvalues.\\
Thus, $\mathcal{E}(G)=2(\lambda_1+\lambda_2+\cdots+\lambda_k)$. Note that $\lambda_1+\lambda_2+\cdots+\lambda_k$ is an eigenvalue of $G\times G\times \cdots \times G$(k-times). We consider two cases based on the presence or absence of middle end vertices. \\
\textbf{Case 1}: If $G$ has no middle end vertices then characteristic polynomial of adjacency matrix is a monic polynomial with integer coefficient, and rational root of such polynomial must be an integer. Hence, $\lambda_1+\lambda_2+\cdots+\lambda_k$ is an integer and  $\mathcal{E}(G)$ is an even integer. \\
\textbf{Case 2}: If $G$ has middle end vertices then adjacency matrix of $G$ contains some entries  as $\frac{1}{2},\,\frac{1}{4}$, so does the adjacency matrix of $G\times G\times \cdots \times G$. Hence, the characteristic polynomial which is monic has coefficients of the form $\frac{p}{2^r}$ for some $p\in \mathbb{Z}, \; r\in \mathbb{N}$. Let ${2^l}$ be the least common multiple of all denominators, multiply the characteristics equation by it to make coefficients integer. Hence, the rational roots are of the form $\frac{N}{2^r}$, for some $N\in \mathbb{Z}, \; r\in \mathbb{N}$. So, if $\lambda_1+\lambda_2+\cdots+\lambda_k$ is a rational number then $\mathcal{E}(G)$ must of the form $\frac{N}{2^r}$, for some $N\in \mathbb{Z}, \; r\in \mathbb{N}$.

\end{proof}
\begin{remark}
When the semigraph $G$ is a graph then we get the graph theory result \cite[Theorem 3.27]{rbp}. 
\end{remark}

\begin{theorem}\label{thm:1}
Let $|V|=n$, $|E|=m$, $r_i$ is the size of the $i^{th}$ edge, then 
$$\displaystyle \mathcal{E}(G)\leq \sqrt{n\left(\frac{1}{6}\sum_{i=1}^{m}r^2_{i}(r^2_i-1)-\frac{15}{8}m_2-\frac{3}{4}m_3-\frac{1}{2}m_4\right)}$$ where $m_2, m_3, m_4$ are defined as earlier. 
\end{theorem}
\begin{proof}
Consider $$\mathcal{E}(G)^2=\left(\sum_{i=1}^{n}|\lambda_i|\right)^2$$
by Cauchy-Schwarz inequality 
$$\mathcal{E}(G)^2=\left(\sum_{i=1}^{n}|\lambda_i|\right)^2 \leq n\sum_{i=1}^{n}|\lambda_i|^2$$
by \cite[lemma 3.3]{pms}
$$\sum_{i=1}^{n}|\lambda_i|^2=\frac{1}{6}\sum_{i=1}^{m}r^2_{i}(r^2_i-1)-\frac{15}{8}m_2-\frac{3}{4}m_3-\frac{1}{2}m_4$$
Hence, combining these two together we get 
$$\mathcal{E}(G)^2\leq n\left(\frac{1}{6}\sum_{i=1}^{m}r^2_{i}(r^2_i-1)-\frac{15}{8}m_2-\frac{3}{4}m_3-\frac{1}{2}m_4\right)$$
Hence, $$\displaystyle \mathcal{E}(G)\leq \sqrt{n\left(\frac{1}{6}\sum_{i=1}^{m}r^2_{i}(r^2_i-1)-\frac{15}{8}m_2-\frac{3}{4}m_3-\frac{1}{2}m_4\right)}.$$

\end{proof}

\begin{remark}
When the semigraph $G$ is a graph, then $m_2=m_3=m_4=0$ and $r_i =2,\; \forall i$. 
Thus, we get \cite[Theorem 5.1]{gutman}
$$\displaystyle \mathcal{E}(G) \leq \sqrt{n\left(\frac{1}{6}\sum_{i=1}^{m} 2^2(2^2-1\right)}=\sqrt{n\left(\frac{1}{6}12m\right)}=\sqrt{2mn}$$
\end{remark}

\begin{theorem} \label{thm:2}
Let $|V|=n$, $|E|=m$, $r_i$ is the size of the $i^{th}$ edge, then 
$$\displaystyle \mathcal{E}(G)\geq \sqrt{\frac{1}{3}\sum_{i=1}^{m}r^2_{i}(r^2_i-1)-\frac{15}{4}m_2-\frac{3}{2}m_3-m_4}$$
 where $m_2, m_3, m_4$ are defined as earlier.

\end{theorem}
\begin{proof}
We know that $trace(A)=0$, implies $\displaystyle \sum_{i=1}^{n}\lambda_i=0.$ Hence $\displaystyle \left(\sum_{i=1}^{n}\lambda_i\right)^2=0.$
Thus, $$\displaystyle \sum_{i=1}^{n}\lambda^2_i+ 2\sum_{i<j}\lambda_i\lambda_j=0$$
by \cite[lemma 3.3]{pms}
$$\displaystyle 2\sum_{i<j}\lambda_i\lambda_j=-\frac{1}{6}\sum_{i=1}^{m}r^2_{i}(r^2_i-1)+\frac{15}{8}m_2+\frac{3}{4}m_3+\frac{1}{2}m_4\;\;\;\;\;\; \cdots (1)$$
As $$\displaystyle \mathcal{E}(G)^2=\sum_{i=1}^{n}|\lambda_i|^2+ 2\sum_{i<j}|\lambda_i\lambda_j|$$
and 
$$\displaystyle  2\sum_{i<j}|\lambda_i\lambda_j|\geq \big{|}2\sum_{i<j}\lambda_i\lambda_j\big{|}$$
implies $$\displaystyle  2\sum_{i<j}|\lambda_i\lambda_j| \geq \frac{1}{6}\sum_{i=1}^{m}r^2_{i}(r^2_i-1)-\frac{15}{8}m_2-\frac{3}{4}m_3-\frac{1}{2}m_4$$
Thus,
$$\displaystyle \mathcal{E}(G)^2\geq \frac{2}{6}\sum_{i=1}^{m}r^2_{i}(r^2_i-1)-\frac{15}{4}m_2-\frac{3}{2}m_3-m_4$$
Hence, 
$$\displaystyle \mathcal{E}(G)\geq \sqrt{\frac{1}{3}\sum_{i=1}^{m}r^2_{i}(r^2_i-1)-\frac{15}{4}m_2-\frac{3}{2}m_3-m_4}.$$
\end{proof}
\begin{remark}
When semigraph $G$ is a graph \cite[Theorem 5.2]{gutman}, we get $$\displaystyle \mathcal{E}(G)\geq 2\sqrt{m}.$$
\end{remark}

\begin{theorem} \label{thm:3}
Let $|V|=n$, $|E|=m$, $r_i$ is the size of the $i^{th}$ edge, then 
$$\displaystyle \mathcal{E}(G)\leq \lambda_1+\sqrt{(n-1)\left(\frac{1}{6}\sum_{i=1}^{m}r^2_{i}(r^2_i-1)-\frac{15}{8}m_2-\frac{3}{4}m_3-\frac{1}{2}m_4-\lambda^2_1\right)}$$ where $m_2, m_3, m_4$ are defined as earlier and $\lambda_1$ is the largest eigenvalue.
\end{theorem}
\begin{proof}
By definition,
$$\displaystyle \mathcal{E}(G)=\sum_{i=1}^{n}|\lambda_i|$$ 
by Cauchy Schwarz inequality 
$$\displaystyle  \sum_{i=2}^{n}|\lambda_i| \leq \sqrt{(n-1)\sum_{i=2}^{n}|\lambda_i|^2}\;\;\;\;\ \cdots (2)$$
by \cite[lemma 3.3]{pms}
$$\sum_{i=1}^{n}|\lambda_i|^2=\frac{1}{6}\sum_{i=1}^{m}r^2_{i}(r^2_i-1)-\frac{15}{8}m_2-\frac{3}{4}m_3-\frac{1}{2}m_4$$
implies 
$$\sum_{i=2}^{n}|\lambda_i|^2=\frac{1}{6}\sum_{i=1}^{m}r^2_{i}(r^2_i-1)-\frac{15}{8}m_2-\frac{3}{4}m_3-\frac{1}{2}m_4-\lambda^2_1$$
Thus, by inequality (2)
$$\displaystyle  \sum_{i=2}^{n}|\lambda_i| \leq \sqrt{(n-1)\left(\frac{1}{6}\sum_{i=1}^{m}r^2_{i}(r^2_i-1)-\frac{15}{8}m_2-\frac{3}{4}m_3-\frac{1}{2}m_4-\lambda^2_1\right)}$$
Hence, 
$$\displaystyle \mathcal{E}(G)\leq \lambda_1+\sqrt{(n-1)\left(\frac{1}{6}\sum_{i=1}^{m}r^2_{i}(r^2_i-1)-\frac{15}{8}m_2-\frac{3}{4}m_3-\frac{1}{2}m_4-\lambda^2_1\right)}$$
\end{proof}
\begin{remark}
When the semigraph $G$ is a graph then we get \cite[Theorem 5.3]{gutman}, $$\displaystyle \mathcal{E}(G)\leq \lambda_1+\sqrt{(n-1)\left(2m-\lambda^2_1\right)}$$
\end{remark}

\subsection{Energies of some special semigraphs}\label{sec:4}
We list down the energies of families of a few semigraphs.
 Let $S^3_{2, n}$ denote a star semigraph having one edge of  3 vertices and n edges of 2 vertices and $S^3_{n}$ represent 3-uniform star semigraph on $2n+1$ vertices and $n$ edges. 
\begin{figure}[h]
\centering
  \begin{tikzpicture}[yscale=0.5]
 \Vertex[size=0.2, y=3, label=$v_2$, position=above, color=black]{B} 
 \Vertex[size=0.2,  label=$v_1$, position=180, color=none]{A} 
 \Vertex[size=0.2,y=-3,  label=$v_3$, position=below, color=black]{C} 
 \Edge(A)(B) \Edge(A)(C)
  \Vertex[size=0.2,x=1.1, y=2.5,  label=$v_4$, position=45, color=black]{D} 
    \Vertex[size=0.2,x=1.3, y=-3, label=$v_r$, position=right, color=black]{E} 
  \Vertex[size=0.2,x=-1.6, y=2.2, label=$v_{n+3}$, position=left, color=black]{F} 
  \Vertex[size=0.2,x=-1.2, y=-3.15, label=$v_{r+1}$, position=left, color=black]{G} 
  \draw[thick](0.2,0.2)--(1.1,2.45);
    \draw[thick](0.15,-0.3)--(1.15,-2.7);
    \draw[thick](-0.2,0.3)--(-1.5, 2.1);
    \draw[thick](-0.2,-0.3)--(-1.2,-3.1);
    \draw[thick](0.2, 0)--(0.2, 0.5);
   \draw[thick](0.15, -0.15)--(0.15, -0.6);
    \draw[thick](-0.2, 0)--(-0.2, 0.5);
   \draw[thick](-0.2, -0.1)--(-0.15, -0.6);
  
  \Edge[bend =20, style={dashed}](D)(E)
   \Edge[bend =20, style={dashed}](G)(F)
   
 \Vertex[size=0.2,x=6, y=3, label=$v_1$, position=above, color=black]{B} 
 \Vertex[size=0.2, x=6, label=$v_0$, position=left, color=none]{A} 
 \Vertex[size=0.2, x=6,y=-3,  label=$v_2$, position=below, color=black]{C} 
 \Edge(A)(B) \Edge(A)(C)
  \Vertex[size=0.2,x=7.1, y=2.5,  label=$v_3$, position=45, color=black]{D} 
    \Vertex[size=0.2,x=7.3, y=-3, label=$v_{2n-1}$, position=right, color=black]{E} 
  \Vertex[size=0.2,x=4.4, y=2.2, label=$v_{2n}$, position=left, color=black]{F} 
  \Vertex[size=0.2,x=4.8, y=-3.15, label=$v_{4}$, position=left, color=black]{G} 
  \Edge(A)(D) \Edge(A)(E) \Edge(A)(F) \Edge(A)(G)
  
  \Edge[bend =20, style={dashed}](D)(E)
   \Edge[bend =20, style={dashed}](G)(F)

\end{tikzpicture}\\
$S^3_{2,n} \hspace{5.5cm} S^3_{n} $
\caption{ }
\label{fig:3}
\end{figure}
We recall the eigenvalues of $S^3_{2,n} $ and $S^3_{n}$. 
\begin{lemma}\cite[Lemma 4.1]{pms}
The spectra of star semigraph $S^3_{2,3}$ are:
$$\begin{pmatrix}0&-2&\lambda_1& \lambda_2& \lambda_3 \\ n-1&1&1&1&1\end{pmatrix}$$ 
 where $\lambda_{1}, \lambda_{2},\lambda_{3}$ are roots of the cubic polynomial $\lambda^3 -2\lambda^2-\frac{n+8}{4}\;+\frac{n}{2}.$
\end{lemma}
\begin{lemma} \cite[Lemma 4.2]{pms}
The spectra of star semigraph $S^3_n$ are:
$$\begin{pmatrix}-2&2& 1-\sqrt{2n+1}& 1+\sqrt{2n+1} \\ n&n-1&1&1\end{pmatrix}$$ 
 
\end{lemma}
Here, we enumerate the eigenvalues of $S^3_{2,n},\; S^3_n, T^3_n\;$ for small values of $n$. All values in the table are rounded to two decimal places.
\begin{table}[h] 
  \begin{center}
  \setlength{\tabcolsep}{10pt} 
\renewcommand{\arraystretch}{2}
    \begin{tabular}{|c|c|c|c|r|} 
    \hline
    n&1&2&3&4\\
      \hline
     $ \mathcal{E}(S^3_{2,n})$ &5.89 &6.20&6.46&6.69\\
       \hline
      $\mathcal{E}(S^3_n)$ &5.46 &10.47 &15.29&\color{blue}{ 20}\\
       \hline
      $\mathcal{E}(T^3_n)$ &5.46  &9.27&12.66& 15.85\\
       \hline
    \end{tabular}
  \end{center}
  \caption*{Energy table}
\end{table}\\

The bound in theorem \ref{thm:3} is tight as it is attained by $S^3_n$ when $n=4$. In $S^3_n$, we have $m_2=0=m_3=m_4$ and when $n=4$ we have $m=4$ edges and 9 vertices. Also, $r_i=4,\; \forall 1\leq i\leq 4$ and largest eigenvalue $\lambda_1$ is 4. Hence, putting these values in \ref{thm:3} we get 
\begin{align*}
\displaystyle \mathcal{E}(S^3_n)\leq& \;4+\sqrt{(9-1)\left(\frac{1}{6}\sum_{i=1}^{4}3^2(3^2-1)-4^2\right)}
\;\;=\; 20
\end{align*}
Thus, energy table confirms that $S^3_n$ attains the bound in theorem \ref{thm:3}. 
\section*{Conclusion}
There is ample scope for study in this topic further. One could study different types of energies associated with the semigarph on parallel lines with graph and see that the results of graph theory are special cases of the results obtained. Also, section \ref{sec:4} opens up an interesting question to study the family of semigraphs which attains the bounds in theorems \ref{thm:1}, \ref{thm:2},  \ref{thm:3}. 

\bibliographystyle{amsplain}
%\bibliography{Adj}

\end{document}